\documentclass[11pt]{amsart}
\usepackage{eucal}
\usepackage{amssymb}
\usepackage{epsfig}
\usepackage{amscd}
\usepackage{enumerate}

\newcommand{\ra}{{\rightarrow}}

\renewcommand{\H}{\mathbb{H}}
\newcommand{\R}{\mathbb{R}}
\newcommand{\C}{\mathbb{C}}

\newcommand{\Z}{\mathbb{Z}}
\newcommand{\N}{\mathbb{N}}

\def\pslc {\mathrm{PSL}(2,\C)}
\def\pslr {\mathrm{PSL}(2,\R)}

\theoremstyle{plain}
\newtheorem{thm}{Theorem}[section]

\newtheorem{lem}[thm]{Lemma}

\newtheorem{dfn}[thm]{Definition}

\numberwithin{equation}{section}
\DeclareMathOperator*{\esssup}{ess\,sup}

\begin{document}
\title{Singular measures on the limit set of a Kleinian group}
       
        \author[Woojin Jeon]{Woojin Jeon}
        \maketitle

\begin{abstract}
We consider a finitely generated torsion free Kleinian group $H$ and a random walk on $H$ with respect to a symmetric nondegenerate probability measure $\mu$ with finite support.
When $H$ is geometrically infinite without parabolics or when $H$ is Gromov hyperbolic with parabolics, we prove that the Patterson-Sullivan measure is singular with respect to the harmonic measure coming from $\mu$.
\end{abstract}
\footnotetext[1]{School of Mathematics KIAS, Hoegiro 87, Dongdaemun-gu, Seoul, 130-722, Korea \\ email : jwoojin@kias.re.kr}
\footnotetext[2]{2010 {\sl{Mathematics Subject Classification.}}
51M10, 57S25, 60J50} \footnotetext[3]{{\sl{Key words and phrases.}}
Kleinian group, Cannon-Thurston map, Poisson boundary, harmonic measure, Patterson-Sullivan measure}
%%============================================================
\section{Introduction}

Let $H$ be a finitely generated {\it Kleinian group}, i.e., a discrete subgroup of $\pslc$. We assume that $H$ has no elliptic elements. Since $\pslc$ is the group of orientation preserving isometries of the hyperbolic 3-space $\H^3$, we can consider the action of $H$ on $\H^3$ which can be continuously extended to the ideal boundary 
$S^2_\infty$ of $\H^3$. The {\it limit set} $\Lambda_H$ of $H$ is the set of limit points of an orbit $H \cdot o$ where $o\in\H^3$. This definition does not depend on the choice of $o$ and $\Lambda_H$ is contained in $S^2_\infty$. 

We consider two measures on $\Lambda_H$. One is the Patterson-Sullivan measure $\rho_x$ based at $x\in \H^3$ and the other is the harmonic measure $\nu_H$ coming from a random walk on $H$. Given a probability measure $\mu$ on $H$ and the corresponding random walk $\{Y_n\}$ on $H$, we can describe $\nu_H$ as follows.
For $A\subset S^2_\infty$, $\nu_H(A)$ is the probability for $Y_n \cdot o$ to converge to a point in $A\subset S^2_\infty$. 
When $H$ is a {\it Fuchsian group} i.e., a discrete subgroup of $\pslr$, $H$ acts isometrically on $\H^2$ and Y. Guivarc'h and Y. Le Jan proved the following.

\begin{thm}\label{Fuchsian}(Y. Guivarc'h and Y. Le Jan \cite{GL})
Let $H$ is a discrete subgroup of $\pslr$ such that $\H^2/H$ is a noncompact surface with finite area.
Let $\nu_H$ is the harmonic measure on $S^1_\infty$ coming from a symmetric nondegenerate probability measure $\mu$ with finite support on $H$. Then $\nu_H$ is singular with respect to the Lebesgue measure on $S^1_\infty$.
\end{thm}

In this paper, 
we generalize Theorem \ref{Fuchsian} to Kleinian groups as follows.

\begin{thm}\label{main}
Suppose that $H$ is a finitely generated Kleinian group such that 
\begin{enumerate}
\item $H$ is Gromov hyperbolic.
\item $H$ is not convex-cocompact.
\item The orbit map $\tau_o:H \ra \H^3$ sending $h$ to $h\cdot o$ extends continuously to $\widehat\tau_o:\partial H \ra S^2_\infty$.
\end{enumerate}
Here $\partial H$ is the Gromov boundary of $H$.
Let $\nu_H$ be the harmonic measure on $\Lambda_H$ coming from a symmetric nondegenerate probability measure $\mu$ with finite support on $H$. Then 
$\nu_H$ is singular with respect to the Patterson-Sullivan measure on $\Lambda_H$.
\end{thm}

We remark that when $\Lambda_H=S^2_\infty$, the Patterson-Sullivan measure $\rho_o$ where $o$ is the origin of the Poincar\'{e} ball model is equal to the Lebesgue measure $Leb_{S^2_\infty}$ on $S^2_\infty$ up to constant multiple. Thus every Patterson-Sullivan measure $\rho_x$ is equal to $Leb_{S^2_\infty}$ up to homothety. The continuous extension $\widehat\tau_o$ of the orbit map $\tau_o$ for $H$ is called the {\it Cannon-Thurston map} of $H$  and
its existence has been verified recently for surface Kleinian groups by \cite{Mj5}. The paper \cite{Mj} dealt with the general case.
For the proof of Theorem \ref{main}, we modify and extend the argument used in the proof of \cite[Prop 5.4, Prop 5.5]{BHM}.
The assumption of the existence of the Cannon-Thurston map is crucial.

\section{Preliminaries}\label{pre}

\subsection{Hyperbolic spaces} For  $\delta \geq 0$, a geodesic metric space $(X,d)$ is called $\delta$-{\it hyperbolic} if for any geodesic triangle in $X$, each side of the triangle is contained in the $\delta$-neighborhood of the union of the other two sides. We call a geodesic metric space  a {\it hyperbolic space} (in the sense of Gromov) if it is $\delta$-hyperbolic for some $\delta\geq 0$.  
A hyperbolic metric space $(X,d)$ has a boundary at infinity $\partial X$ called the {\it Gromov boundary} which can be defined as follows. We say a sequence $\{x_n\}$ in $(X,d)$ {\it converges to infinity} if $$\liminf_{i,j\ra \infty}(x_i \vert x_j)_x=\infty$$ for some (hence every) basepoint $x$, where
$$(y \vert z)_x:= {1 \over 2}(d(x,y)+d(x,z)-d(y,z))$$
for $x,y,z \in X$. This product $(y \vert z)_x$ is called  the {\it Gromov product} of $y$ and $z$ with respect to $x$ and approximates within $2\delta$ the distance from $x$ to any geodesic $[y,z]$ joining $y,z$.
Two sequences $\{x_n\}$, $\{y_n\} \subset X$ converging to infinity are said to be {\it equivalent} if $$\liminf_{i,j\ra \infty}(x_i \vert y_j)_x=\infty$$ for some $x\in X$.
The Gromov boundary $\partial X$ is defined as the set of the equivalence classes of sequences converging to infinity in $X$. This definition is independent of the choice of the base point $x\in X$.
The Gromov product can be extended to $\partial X$. We can define a natural metric $d_\epsilon(\cdot,\cdot)$ on $\partial X$ 
such that $C^{-1}e^{-\epsilon(p\vert q)_x} \leq d_\epsilon(p,q) \leq Ce^{-\epsilon(p\vert q)_x}$ where $(p,q) \in \partial X\times \partial X$.

Fixing a finite generating set $S$ of $H$, the word metric $d_w(\cdot,\cdot)$ on $H$ is defined by setting $d_w(g, h)=\vert g^{-1}h\vert_S$ where $\vert g\vert_S$ is the minimum of the number of elements of $S$ whose product is $g$. 
When we consider the right Cayley graph $\Gamma_H$ of $H$ with the length of every edge being $1$, $d_w(g,h)$ is the minimum of the lengths of paths joining $g$ and $h$ in $\Gamma_H$.
A finitely generated group $H$ is called as a {\it hyperbolic group} if its Cayley graph $\Gamma_H$ is hyperbolic with respect to the word metric for a finite generating set $S$ of $H$. We simply denote the word length of $g$ as $\vert g\vert$ by fixing a finite generating set. We denote the Gromov boundary of $\Gamma_H$ by $\partial H$.
For basic properties of Gromov hyperbolic spaces, hyperbolic groups and their Gromov boundaries, we refer the reader to \cite{Gro, GH} and \cite{KB}.

\subsection{Kleinian groups}
%We refer to \cite{CB, Thu} for basic theory of Kleinian groups and geodesic laminations. 
Let $H$ be a finitely generated {\it Kleinian group}. We assume that $H$ has no elliptic elements. 
The {\it limit set} $\Lambda_H$ of $H$ is the set of limit points of an orbit $H \cdot o$ where $o\in\H^3$. 
The complementary open set $\Omega_H=S^2_\infty\backslash \Lambda_H$ becomes the domain of discontinuity of the action of $H$ and by Ahlfors' finiteness theorem, $\Omega_H/H$ consists of finitely many Riemann surfaces which are called as the {\it conformal boundaries at infinity} of $H$.

%The hyperbolic 3-manifold $M=\H^3/H$ has finitely many ends and each end $E$ is homeomorphic to $S_E\times[0, \infty)$ for some hyperbolic surface $S_E$ by the Tameness Theorem\cite{Ag, CG}.
%An end $E$ of $M$ is called {\it geometrically infinite} if there exists an exiting sequence $\{c_n\}$ of closed geodesics in $M$ and $E$ is called {\it geometrically finite} otherwise.
%A geometrically finite end corresponds to a Riemann surface which is an element of the conformal boundary at infinity of $H$. For a geometrically infinite end $E$, we can define an {\it ending lamination} $\lambda_E$ as follows.
%A {\it geodesic lamination} $\lambda$ is by definition a closed union of simple geodesics on a hyperbolic manifold.
%For a sequence $\{c_n\}$ of closed geodesics exiting $E$, we homotope each $c_n$ to $c_n'$ which is a closed geodesic in $S_E$.
%Since the space of geodesic laminations is compact with respect to the Hausdorff topology, there exists a subsequential limit $\lambda_E$ of $\{c_n'\}$ on $S_E$ which is independent of the choice of the exiting sequence $\{c_n\}$.

The {\it convex hull} of $H$ is defined to be the smallest convex set in $\H^3$ whose closure in $\H^3\cup S^2_\infty$ containes $\Lambda_H$. The {\it convex core} of $H$ is the quotient of its convex hull by $H$ itself.
The Kleinian group $H$ is called {\it geometrically finite(resp. convex cocompact)} if its convex core has finite volume(resp. if its convex core is compact).
Every convex cocompact Kleinian group is a hyperbolic group by the Schwarz-Milnor lemma(See \cite{Bri} for example).
By Thurston's uniformization theorem in \cite{Mor}, every geometrically infinite Kleinian group without parabolics has a convex cocompact
representation which is faithful. Thus if $H$ is geometrically infinite without parabolics, then it is hyperbolic. 

If the orbit map $\tau_o :H \ra \H^3$ sending $h \ra h\cdot o$ can be extended continuously to $\widehat\tau_o :\partial H \ra S^2_\infty$, we call $\widehat\tau_o$ as the {\it Cannon-Thurston map} of $H$ following \cite{CT, Bo1, Mc, Min, Mj5, Mj, Mj7}.
%When $H$ is not geometrically finite and has no parabolics, we have an ending lamination $\lambda_E$ corresponding to a geometrically infinite end $E$ of $\H^3/H$. Since we have a convex cocompact representation $\rho :H \ra \pslc$, 
%the Cayley graph $\Gamma_H$ and the the convex hull of the limit set $\Lambda_{\rho(H)}$ of $\rho(H)$ are quasi-isometric by the orbit map $\tau_\rho$ sending $h \ra \rho(h)\cdot o$.
%Regarding the exiting closed geodesic $c_n$ as an element in $\pi_1(\H^3/H)=H$, we consider a subsequential Hausdorff limit $\lambda_E'$ of the closed geodesics corresponding to $\rho(c_n)$. A leaf $L$ of $\lambda_E'$ is a bi-infinite geodesic in $\H^3/\rho(H)$ so $\lambda_E'$ is a geodesic lamination in $\H^3/\rho(H)$.
%A lift $\widetilde L\subset \H^3$ of $L$ is quasi-isometric to a bi-infinite quasigeodesic $l$ in $\Gamma_H$ by $\tau_\rho$.
%Assuming $\widehat\tau_o$ exists, 
%if $p,q \in \partial H$ are two endpoints of $l$, it is known that $\widehat\tau_o$ identifies $p$ and $q$, i.e., we have $\widehat\tau_o(p)=\widehat\tau_o(q)$.
If $H$ is convex cocompact, then $\tau_o$ is a quasi-isometric embedding and we have a natural homeomorphism $\widehat \tau_o$ from the Gromov boundary $\partial H$ to the limit set $\Lambda_H$. 
When $H$ is a geometrically finite Kleinian group with parabolics, we still have a continuous extension of $\tau_o$ from the Floyd boundary (See \cite{Flo, Ger}) of $H$ to $\Lambda_H$. If we further assume $H$ is hyperbolic, then the Floyd boundary of $H$ is equal to the Gromov boundary $\partial H$ and a parabolic element $h$ gives two points $\{h_\infty, h_{-\infty}\}$ in $\partial H$ such that $\widehat\tau_o(h_\infty)=\widehat\tau_o(h_{-\infty})$ where $\{h_\infty, h_{-\infty}\}$ are the accumulation points of $\{h^i\vert i\in \Z\}$ in $\Gamma_H$.  

Now we consider the case when $H$ is a geometrically infinite hyperbolic Kleinian group and we assume the Cannon-Thurston map $\widehat\tau_o$ exists.
In this case, we can find an exiting sequence of closed geodesics $\{c_n\}$ in $\H^3/H$.
The exiting property of $\{c_n\}$ gives a bi-infinite quasigeodesic $l$ in $\Gamma_H$ such that 
the two endpoints of $l$ are identified by $\widehat\tau_o$. See \cite{Mj6} for details.
We need the following lemma for the proof of Theorem \ref{main}.

\begin{lem}\label{exp}
Let $H$ be a finitely generated Kleinian group with $\H^3/H$ not being convex cocompact and such that $H$ is Gromov hyperbolic.
We assume that the Cannon-Thurston map $\widehat\tau_o:\partial H \ra S^2_\infty$ exists.
Then there exists $\{h_n\}\subset H$ such that for any constant $D>0$, $\vert h_n \vert -Dd_{\H^3}(o, h_n\cdot o) \ra \infty$.
\end{lem}
\begin{proof}
We fix a finite generating set $S$ of $H$.
If $H$ has a parabolic element $g$, then we can just take $h_n$ as $g^n$.
If not, then $H$ is geometrically infinite and there exists a bi-infinite quasigeodesic $l$ in $\Gamma_H$ joining two points $p,q$ in $\partial H$ such that $\widehat\tau_o(p)=\widehat\tau_o(q)$.
We represent $l$ as a bi-infinite sequence of vertices in $\Gamma_H$ 
$$\{\cdots g_{-n}, g_{-(n-1)}, \cdots, id, \cdots, g_n, g_{n+1},\cdots\}$$ 
such that $\vert g_i^{-1}g_{i+1}\vert =1$.
Then since $\widehat\tau_o(p)=\widehat\tau_o(q)$ and $\widehat\tau_o$ is a continuous extension of $\tau_o$, we have $(g_{-n}\cdot o\vert g_n\cdot o)_o \ra \infty$. If $\tau_o\vert_l$ is a quasi-isometric embedding then by the Morse-Mostow lemma(See \cite{Bri} for example), the image of $\tau_o\vert_l$ is contained in a uniform  neighborhood of a bi-infinite geodesic in $\Gamma_H$. Since $(g_{-n}\cdot o\vert g_n\cdot o)_o \ra \infty$, this is not possible and thus $\tau_o\vert_l$ cannot be a quasi-isometric embedding.
We note that $d_{\H^3}(g, h)\leq C\vert g^{-1}h\vert$ by the triangle inequality where $C:=\max\limits_{s\in S}d_{\H^3}(o, s\cdot o)$.
We can also see the following.
There exists a sequence of integers $i(n)$ and $j(i)>i$ such that for any $D>0$, $\vert g_i^{-1}g_j \vert-Dd_{\H^3}(g_i\cdot o, g_j\cdot o)  \ra \infty$ as $n\ra\infty$. We take $h_n$ as $g_{i(n)}^{-1}g_{j(i)}$. 

\end{proof}

\subsection{Poisson boundary}
Let $\mu$ be a probability measure on a group $H$.
\begin{itemize}
\item $\mu$ is called {\it nondegenerate} if the support of $\mu$ generates $H$ as a semigroup.
\item $\mu$ is called {\it symmetric} if $\mu(A)=\mu(A^{-1})$ where $A\subset H$ and $A^{-1}$ is the set of the inverses of elements of $A$.
\end{itemize}
The word length $\vert h \vert$ of an element $h$ in $H$ is with respect to a fixed finite generating set of $H$. The {\it first moment} of $(H,\mu)$ is defined to be $\sum\limits_{h\in H}\vert h\vert\mu(h)$. The {\it entropy} of $(H,\mu)$ is $\lim\limits_{n \to \infty}-\sum\limits_{h\in H}\mu(h)\log\mu(h)$. When $\mu$ has a finite support, both of its entropy and its first moment are finite.

The {\it random walk} on $H$ with respect to $\mu$ is a Markov chain $\{Y_n\}$ with the transition probabilities given by $p_{h_1,h_2}=\mu(h_1^{-1}h_2)$. We denote $\{Y_n\}$ by $Y$. A random walk $Y$ can also be described as a sequence of independent random variables $\{Z_n\}$ with values in the probability space $(H,\mu)$ such that $Y_n=Z_0Z_1\cdots Z_n$. When $Y_0=id$, the random walk starts from the identity element $id$ of $H$.
We regard $Z=\{Z_n\}$ as an element of the product probability space $(H^\N,\mu^\N)$ and $Y$ as an element of the Kolmogrov representation space
$(H^\N, \mathbb{P}_{id})$ where $$\mathbb{P}_{id}(\{Y_n\} \vert Y_0=id, Y_1=h_1, Y_2=h_2, \cdots, Y_n=h_n\})=p_{id,h_1} p_{h_1,h_2} \cdots p_{h_{n-1}, h_n}$$
and $\mathbb{P}_{id}(\{Y_n\} \vert Y_0\neq id \})$ is defined to be zero.
We denote $\mathbb{P}_{id}$ by $\mathbb{P}$.

The time shift operator $T$ acts on $(H^\N,\mathbb{P})$ by $(TY)_n=Y_{n+1}$ and defines an equivalence relation $\thicksim$ by saying $Y\thicksim Y'$ if and only if there exist positive integers $k, k'$ such that $T^kY=T^{k'}Y'$.
\begin{dfn}
The Poisson boundary of $H$ with respect to $\mu$ is the quotient space of $(H^\N, \mathbb{P})$ by the smallest measurable equivalence relation generated by the equivalence relation $\thicksim$. 
\end{dfn}
The smallest measurable equivalence relation generated by $\thicksim$ is called as the {\it measurable envelope} \cite{Roh}. 
When $H$ is a hyperbolic group, we can consider the hitting measure $\nu$ of a random walk $Y$ on $\partial H$ and the Poisson boundary of $H$ can be identified with $(\partial H, \nu)$  by \cite{Kai1, Kai2}. More precisely, it can be shown that
$\{Y_n\}$ converges to a point in $\partial H$ for almost every sample path $\{Y_n\}$ and $\nu(A)$ can be defined as the probability for a sample path $\{Y_n\}$ converges to a point in $A\subset \partial H$. Thus $\nu$ can also be called as {\it the law of $Y_\infty$}.
When $H$ is a hyperbolic Kleinian group, $H$ acts on $\H^3$ and we can consider the law of $Y_\infty\cdot o$ if $\{Y_n\cdot o\}$ converges to a point in $S^2_\infty$ for almost every sample path $\{Y_n\}$. In fact, the law of $Y_\infty\cdot o$ is the push forward of $\nu$ by $\widehat\tau_o$ by the following Theorem \ref{KM} which summarizes \cite[Corollary 6.2]{KM}.

\begin{thm}(Karlsson and Margulis \cite{KM})\label{KM}
Let $H$ be a countable group isometrically acting on a uniformly convex, complete metric space $(X,d)$ which is nonpositively curved in the sense of Busemann. Let $\mu$ be a probability measure on $H$ with $\sum\limits_{h\in H}d(o, h\cdot o)\mu(h)<\infty$. Let $\partial X$ be the ideal boundary of $X$ consisting of asymptotic classes of geodesic rays. Then 
\begin{itemize}
\item Almost every sample path $\{Z_n\}$ in $H^\N$ with respect to $\mu^\N$ converges to a geodesic ray in $X$.
\end{itemize}
Thus we have a map $\xi:(H^\N, \mu^\N) \ra \partial X$ and we give $\partial X$ the push-forward measure $\xi_*(\mu^\N)$ so that $\xi$ is measurable.
\begin{itemize}
\item $(\partial X, \xi_*(\mu^\N))$ becomes the Poisson boundary of $(H, \mu)$ if the lattice counting function is subexponential, i.e., if $$\#\{h\in H \vert d(o,h\cdot o)<r\} \leq e^{Cr}$$ for some constant $C>0$.
\end{itemize}
\end{thm}

A similar form of Theorem \ref{KM} was also mentioned in Remark 3 following \cite[Theorem 7.7]{Kai2}.
It is known that the lattice counting function is subexponential for a large class of discrete subgroups of the isometry groups of Cartan-Hadamard manifolds. In particular, for pinched negatively curved case, it is proven in \cite[Theorem 3.6.1]{Yue}. Thus every Kleinian group has a subexponential lattice counting function.
Moreover, since we are assuming $H$ is finitely generated, we have $d_{\H^3}(o, h\cdot o) \leq D\cdot \vert h\vert$ by the triangle inequality.

Therefore we can apply Theorem \ref{KM} to a Kleinian group $H$ acting on $\H^3$. We assume $\mu$ is a symmetric, nondegenerate probability measure on $H$ with finite support. Let $\nu_H$ be the law of $Y_\infty\cdot o$. Then $(S^2_\infty, \nu_H)$ becomes the Poisson boundary of $(H,\mu)$ and thus $\widehat \tau_o :(\partial H, \nu) \ra (S^2_\infty, \nu_H)$ is a measurable isomorphism.

\subsection{Conformal density}
For $x$ in $\H^3$,
the {\it Busemann function} $b_{x, \eta}(\cdot)$ at $\eta$ with $b_{x, \eta}(x)=0$ can be defined by choosing a geodesic ray $\alpha(t)$ from $\alpha(0)=x\in \H^3$ toward $\eta\in S^2_\infty$ as follows.
$$b_{x, \eta}(y)=\lim\limits_{t\ra \infty}(d_{\H^3}(y, \alpha(t))-d_{\H^3}(x, \alpha(t)))$$
\begin{lem}(\cite[Lemma 3.2.1]{Nic})\label{busemann}
Consider the Poincar\'{e} ball model $\{x\in \R^3: \vert x\vert<1\}$ of $\H^3$ where $\vert x\vert$ is the usual Euclidean norm of $x$. Then
$$e^{b_{x,\eta}(y)}=\frac{P(x, \eta)}{P(y, \eta)}$$
where $P(x,\eta)$ is the Poisson kernel $(1-\vert x\vert^2)/\vert x-\eta\vert^2$.
Thus $b_{x,\eta}(y)$ is a continuous function of $x,y\in \H^3$ and $\eta\in S^2_\infty$.
\end{lem}
Fixing $x,y$ in $\H^3$, the Poincar\'{e} series for $H$ is 
$$g_s(x,y)=\sum_{h \in H} e^{-sd_{\H^3}(x,h\cdot y)}$$
The critical exponent $\delta_H$ of $H$ is defined as
$$\delta_H=\limsup\limits_{r \ra \infty}\frac{1}{r}\log(\#\{h\in H \vert d_\H^3(o, h\cdot o)\leq r\})$$
Equivalently, $\delta_H$ is the infimum of the set of $s$ such that $g_s(x,y)$ is finite.
We call $H$ is {\it divergent} if $g_s(x,y)$ diverges at $s=\delta_H$. Otherwise $H$ is called {\it convergent}.
These definitions are independent of the choices of $x,y\in \H^3$.
Consider a family of measures $\{\rho_x^s\}$ defined by $$\rho_x^s=\frac{1}{g_s(y,y)}\sum_{h \in H} e^{-sd_{\H^3}(x,h\cdot y)}\delta_{h\cdot y}$$ where $\delta_{h\cdot y}$ is the Dirac measure at $h\cdot y$. Then we can find a sequence $\{s_i\}$ with $s_i \ra \delta_H^{-}$ such that $\rho_x^{s_i}$ weakly converges to a finite measure $\rho_x$ on the compact space $\H^3\cup S^2_\infty$.
When $H$ is divergent, $\rho_x$ has its support on the limit set $\Lambda_H$. Since $\{\rho_x\}$ satisfies $$\frac{d\rho_x}{d\rho_y}(\eta)=e^{-\delta_H b_{x,\eta}(y)},  h^{*}\rho_x=\rho_{h^{-1}\cdot x}$$
it becomes an $H$-invariant {\it conformal density} of dimension $\delta_H$. 
Here the pull-back measure $h^*\rho_x$ is defined by setting $(h^*\rho_x)(A):=\rho_x(h\cdot A)$ for $A\subset \Lambda_H$.
Even when $H$ is convergent, we still can construct a conformal density by increasing the Dirac mass on each orbit point suitably \cite{Pa, Sul} and we denote a resulting conformal density also by $\rho_x$.
We call $\rho_x$ as the {\it Patterson-Sullivan measure} with base point $x\in \H^3$ in either case of $H$ being convergent or divergent.

As an application of the Tameness theorem(see section 9 in \cite{Ca} and \cite[Prop. 3.9]{CS}), a Kleinian group $H$ is divergent if and only if either $H$ is geometrically finite or $\Lambda_H=S^2_\infty$. When $\Lambda_H=S^2_\infty$, $\rho_x$ is equal to $Leb_{S^2_\infty}$ up to homothety and the diagonal action of $H$ on $S^2_\infty\times S^2_\infty\backslash\Delta$ is ergodic with respect to $Leb_{S^2_\infty}\otimes Leb_{S^2_\infty}$. Here $\Delta$ means the diagonal set.

\subsection{Harmonic Density}
The {\it Green function} on a group $H$ with a probability measure $\mu$ is defined by 
$$G(g, h)=\sum\limits_{n=0}^{\infty}\mu^n(g^{-1}h)$$
where $\mu^n$ is the $n$-th convolution power of $\mu$.
Then $F(g,h)=G(g,h)/G(h,h)$ is the probability that there exists $n\in \N$ such that $gY_n=h$.
The Green metric $d_G$ is defined by $$d_G(g,h)=-\log F(g,h)$$ 
If $H$ is a hyperbolic group and
if $\mu$ is finitely supported nondegenerate symmetric probability measure on $H$, then $d_G$ is a left invariant hyperbolic metric quasi-isometric to the word metric on $H$ by \cite[Corollary 1.2]{BHM}.

The {\it Martin kernel} $K:H\times H \ra \R$ is defined by 
$$K(g,h)=\frac{F(g,h)}{F(id,h)}$$
There exist constants $\{C_g\}$ such that for all $h\in H$, $K(g,h)\leq C_g$ for each $g\in H$. Let $H\cup \partial_M H$ be the metric completion of $H$ with respect to the metric $d_M$ defined by $$d_M(h_1,h_2)=\sum\limits_{h\in H}D_h\frac{\vert K(h, h_1)-K(h, h_2)\vert+\vert \delta_{h,h_1}-\delta_{h, h_2}\vert}{C_h+1}$$
where $\{D_h\}$ is chosen so that $\sum\limits_{h\in H}D_h<\infty$. $\delta_{h,g}$ is the Kronecker delta.
Then it can be shown that $H\cup \partial_M H$ is in fact a compactification of $H$ and the Martin kernel can be continuously extended 
to $H\times (H\cup \partial_M H)$. We call $\partial_M H$ as the {\it Martin boundary} of $H$. When $H$ is a hyperbolic group, $\partial_M H$ is homeomorphic to the Gromov boundary of $H$.
For this, see \cite{An, Kai3} or \cite[Corollary 1.8]{BHM}. 

A positive function $u$ on $H$ is called {\it harmonic} if $u(h)=\sum\limits_{g\in H}\mu(h^{-1}g)u(g)$ for all $h\in H$ and the 
Martin representation theorem says for any harmonic function $u$ on $H$, there exists a measure $\nu_u$ on $\partial_M H$ such that $$u(h)=\int_{\partial_M H}K(h, \xi)d\nu_u(\xi)$$
When we take $u$ as a constant function $u\equiv 1$, the support of $\nu_u$ with the measure $\nu_u=\nu_1$ becomes the Poisson boundary of $H$. 
It is known that almost every sample path $\{Y_n\}$ converges to a point on $\partial_M H$.
Thus for a hyperbolic group $H$, we can identify $\nu$ with $\nu_1$.
We have the following change of variable formula.
$$\frac{d(h^*\nu)}{d\nu}(\xi)=K(h^{-1}, \xi)$$
For a hyperbolic group $H$, 
we can define the Busemann function on $H$ with respect to the Green metric $d_G$ by 
$$b^G_{id, \xi}(h)=\sup\limits_{\{h_n\}}\limsup\limits_{n\ra \infty}(d_G(h_n, id)-d_G(h_n, h))$$ where $\sup$ is taken over all sequences $\{h_n\}$ converging to $\xi$.
If we choose $\{h_n\}$ along a quasigeodesic ray toward $\xi$ in $\Gamma_H$, then we can just take the usual limit as $h_n\ra \xi$ to define $b^G_{id, \xi}(h)$.
By the definition of the Martin kernel, we get $$\frac{d(h^*\nu)}{d\nu}(\xi)=e^{b^G_{id, \xi}(h^{-1})}$$ 
We have the same formula for the push-forward measure $\nu_H$ on $S^2_\infty$ which was used in the proof of \cite[Prop. 5.5]{BHM}
\begin{lem}
Let $H$ be a Kleinian group which is hyperbolic and
let $\nu_H$ be the push-forward measure $\widehat\tau_o(\nu)$ on $S^2_\infty$ as before. Then
$$\frac{d(h^*\nu_H)}{d\nu_H}(\widehat\tau_o(\xi))=e^{b^G_{id, \xi}(h^{-1})}$$
where $\xi\in \partial H$ and $(h^*\nu_H)(A)=\nu_H(h\cdot A)$ for $A\subset S^2_\infty$ for any $\nu_H$-measurable set $A\subset S^2_\infty$.
\end{lem}

Now we apply \cite[Theorem 3.3]{Kai3} to the Cayley graph of the hyperbolic group $H$.
Since the Gromov boundary $\partial H$ itself is the conical limit set of $H$, we can see that the diagonal action of $H$ on $\partial H\times \partial H\backslash\Delta$ with respect to the measure class of $\nu \otimes \nu$ is ergodic. 
\begin{lem}\label{cor}
Let $H$ be a Kleinian group which is hyperbolic. Let $\mu$ be a symmetric nondegenerate probability measure on $H$ with finite support.
Then the diagonal action of $H$ on $S^2_\infty\times S^2_\infty\backslash \Delta$ is ergodic with respect to $\nu_H\otimes \nu_H$. 
\end{lem}
\begin{proof}
We know $(\partial H, \nu)$ and $(S^2_\infty, \nu_H)$ are measurably isomorphic by the equivariant map $\widehat\tau_o$ by Theorem \ref{KM}.
\end{proof}

\section{Singularity of measures on $\Lambda_H$}
In this section, we prove Theorem \ref{main}.
Recall that for $\eta_1, \eta_2 \in S^2_\infty$, the {\it Busemann cocycle} $B_o(\eta_1, \eta_2)$ is defined as $b_{o,\eta_1}(y)+b_{o,\eta_2}(y)$ where $y$ is any point in the bi-infinite geodesic $l$ joining $\eta_1$ and $\eta_2$. Geometrically $B_o(\eta_1, \eta_2)$ is the length of the geodesic subsegment of $l$ contained in the intersection of the horoballs passing through $o$ and centered at $\eta_1, \eta_2$. Thus it is nonnegative for any $(\eta_1,\eta_2)\in S^2_\infty\times S^2_\infty\backslash\Delta$ and it is zero if and only if $o$ is contained in $l$.

\begin{proof}[Proof of Theorem \ref{main}] 
We follow and extend the proof of \cite[Prop. 5.4, Prop. 5.5]{BHM}.
Define the measure $\widetilde\rho_o$ on $S^2_\infty\times S^2_\infty\backslash \Delta$ by 
$$d\widetilde\rho_o(\eta_1, \eta_2)=e^{2\delta_H B_o(\eta_1, \eta_2)}d\rho_o(\eta_1)  d\rho_o(\eta_2)$$
Then $\widetilde\rho_o$ is a $H$-invariant measure although it may not be an ergodic measure with respect to the $H$-action.
We define the measure $\widetilde\nu_H$ on $S^2_\infty\times S^2_\infty\backslash\Delta$ as 
$$d\widetilde\nu_H(\widehat\tau_o\xi_1, \widehat\tau_o\xi_2)=e^{2(\xi_1\vert \xi_2)^G_{id}}d\nu_H(\widehat\tau_o\xi_1)  d\nu_H(\widehat\tau_o\xi_2)$$
Then $\widetilde\nu_H$ is invariant under the action of $H$ by \cite[Prop. 2.2]{Kai3}. Furthermore it is an ergodic measure by Corollary \ref{cor}. Now
we suppose $\nu_H$ is equivalent to $\rho_o$ and we claim that the Radon-Nikodym derivative $d\rho_o/d\nu_H$ is $\nu_H$-essentially upper and lower bounded by positive constants.

\noindent{\it Proof of the claim} : We have a $\nu_H$-integrable function $J$ and a $\widetilde\nu_H$-measurable function $\widetilde J$ defined by $d\nu_H=Jd\rho_o$ and $d\widetilde\nu_H=\widetilde J d\widetilde \rho_o$.
Since $\widetilde J$ is positive almost everywhere, there exists a constant $C>0$ such that the set
$A:=\{(\eta_1, \eta_2)\in S^2_\infty\times S^2_\infty\backslash\Delta \vert C^{-1}\leq \widetilde J(\eta_1, \eta_2) \leq C\}$ has positive $\widetilde\nu_H$-measure.
Since $\widetilde\nu_H$ is ergodic, there exists $h\in H$ such that $(h\eta_1, h\eta_2)\in A$ for $\widetilde\nu_H$-almost every $(\eta_1, \eta_2)$. Since $\widetilde\rho_o$ is also $H$-invariant, we have $\widetilde J(\eta_1, \eta_2)=\widetilde J(h \eta_1, h\eta_2)$ and thus $A$ has the full $\widetilde\nu_H$-measure.
When $\eta_i=\widehat\tau_o\xi_i$ for $i=1,2$, we have 
$$\widetilde J(\eta_1, \eta_2)=J(\eta_1)J(\eta_2)\frac{ e^{2(\xi_1\vert \xi_2)^G_{id}}}{e^{2\delta_H B_o(\eta_1, \eta_2)}}$$
Thus we have for $\widetilde\nu_H$-almost every $(\eta_1, \eta_2)$, 
$$C^{-1}\frac{e^{2\delta_H B_o(\eta_1, \eta_2)}}{e^{2(\xi_1\vert \xi_2)^G_{id}}}\leq  J(\eta_1)J(\eta_2) \leq C\frac{e^{2\delta_H B_o(\eta_1, \eta_2)}}{ e^{2(\xi_1\vert \xi_2)^G_{id}}}$$
Assume, seeking a contradiction, $J$ is unbounded in $B_s\cap \Lambda_H \subset S^2_\infty$ where $B_s$ is a small spherical open ball in $S^2_\infty$.
We choose another small spherical open ball $B_s'$ far from $B_s$ so that for all $(\eta,\eta')\in B_s\times B_s'$,
$D^{-1}<e^{B_o(\eta, \eta')}<D$ for some constant $D>0$. 
Since $\widehat\tau_o$ is a continuous map, the distance in $\partial H$ from the open set $\widehat\tau_o^{-1}(B_s')$ to the open set $\widehat\tau_o^{-1}(B_s)$ has a positive lower bound.
If we let $\widehat\tau_o\xi=\eta$ and $\widehat\tau_o\xi'=\eta'$,
this means $(\xi' \vert \xi)_{id}$ has an upper bound and the same is true for $(\xi'\vert \xi)^G_{id}$. Thus $e^{(\xi'\vert \xi)^G_{id}}$ is positively upper and lower bounded.
From this, we get that $J(\eta)J(\eta')$ is positively upper and lower bounded for 
$\widetilde\nu_H$-almost all $(\eta, \eta')\in B_s\times B_s'$. But there is a constant $D_1>0$ such that 
$B:=\{\eta'\in B_s' \vert D_1^{-1} \leq J(\eta') \leq D_1\}$ has a positive measure so that $B_s\times B$ has a positive $\widetilde\nu_H$-measure. Since $J$ is unbounded on $B_s$, $J(\eta)J(\eta')$ cannot be essentially bounded in $B_s\times B$. This is a contradiction and we have proved our original claim.

Now we recall the change of variable formulas for conformal and harmonic measures.
$$\frac{d(h^*\rho_o)}{d\rho_o}(\eta)=e^{\delta_H b_{o,\eta}(h^{-1}\cdot o)}$$
$$\frac{d(h^*\nu_H)}{d\nu_H}(\widehat\tau_o(\xi))=e^{b^G_{id, \xi}(h^{-1})}$$
Since the density of $\rho_o$ with respect to $\nu_H$ is uniformly bounded away from zero, there exists $C_1>0$
$$\vert\esssup\limits_{\eta\in \Lambda_H} \delta_H b_{o,\eta}(h^{-1}\cdot o)- \esssup\limits_{\xi\in\partial H}b^G_{id, \xi}(h^{-1})\vert<C_1$$
By \cite[Lemma 2.5]{BHM}, there exists a constant $C_2>0$ such that
$$\vert d_G(id, h^{-1})-\esssup\limits_{\xi\in\partial H}b^G_{id, \xi}(h^{-1}) \vert<C_2$$
Here we can replace `esssup' by `sup' because every open set in $\partial H$ has a positive $\nu$-measure
and $b^G_{id, \xi}(h^{-1})$ is a continuous function with respect to $\xi$ by Lemma \ref{busemann}. 
We also have
$$\sup\limits_{\eta\in \Lambda_H}b_{o, \eta}(h^{-1}\cdot o)\leq d_{\H^3}(o, h^{-1}\cdot o)$$
by the triangle inequality.
Therefore there exist $C_3$(which may not be a positive number) such that for all $h\in H$,
$$\delta_H d_{\H^3}(o, h^{-1}\cdot o)- d_G(id, h^{-1})>C_3$$
But by Lemma \ref{exp}, there exists a sequence $\{h_n\}\subset H$ such that $\vert h_n^{-1}\vert -\delta_H d_{\H^3}(o, h_n^{-1}\cdot o)$ goes to infinity. Since the word metric is quasi-isometric to the Green metric, we have a contradiction.
\end{proof}

Note that even in the case of $\Lambda_H=S^2_\infty$, Theorem \ref{main} is not a direct consequence of \cite[Prop. 4.5]{BHM}.
In fact, \cite[Prop. 4.5]{BHM} is using \cite[Prop. 4.4]{BHM} which is valid for the case that $\tau_o$ is a quasi-isometry and $\widehat\tau_o$ is the natural homeomorphism between $\partial H$ and $S^2_\infty$. For the case that $\tau_o$ is not a quasi-isometry, we need to assume the existence of 
the continuous boundary extension $\widehat\tau_o$ of $\tau_o$.

\space Woojin Jeon\\ School of Mathematics\\
 KIAS, Hoegiro 87, Dongdaemun-gu\\
 Seoul, 130-722, Korea\\
\texttt{jwoojin\char`\@ kias.re.kr}\\

\end{document}